\documentclass[a4paper, 12pt]{amsart}

\usepackage[margin=2.5cm]{geometry}   

\usepackage[mnsy]{MnSymbol}
\usepackage{verbatim} 
\usepackage{graphicx}
\usepackage{psfrag}
\usepackage{pinlabel}

\newcommand{\R}{\mathbb R}

\newcommand{\Q}{\mathbb Q}
\newcommand{\Z}{\mathbb Z}

\newcommand{\Diff}{\rm Diff}

\newcommand{\al}{\alpha}

\newcommand{\de}{\delta}

\newcommand{\ga}{\gamma}

\newcommand{\Si}{\Sigma}

\renewcommand{\th}{\theta}

\newcommand{\x}{\times}

\newcommand{\del}{\partial}
\newcommand{\co}{\thinspace\colon}
\DeclareMathOperator{\End}{End}

\newtheorem{thm}{Theorem}[section]
\newtheorem{lemma}[thm]{Lemma}
\newtheorem{cor}[thm]{Corollary}

\newtheorem{prop}[thm]{Proposition}

\theoremstyle{definition}
\newtheorem{defn}[thm]{Definition}

\begin{document}
\author{Paolo Ghiggini and Paolo Lisca}
\title[Open book decompositions versus prime factorizations]
{Open book decompositions versus prime factorizations of closed, oriented 3--manifolds}

\begin{abstract}
Let $M$ be a closed, oriented, connected 3--manifold and $(B,\pi)$ an open book decomposition on $M$ with 
page $\Si$ and monodromy $\varphi$. It is easy to see that the first Betti number of $\Si$ is bounded below 
by the number of $S^2\times S^1$--factors in the prime factorization of $M$. Our main result is that equality is realized if 
and only if $\varphi$ is trivial and $M$ is a connected sum of $S^2\times S^1$'s. We also give some applications 
of our main result, such as a new proof of the result by Birman and Menasco that if the closure of a braid 
with $n$ strands is the unlink with $n$ components then the braid is trivial. 
\end{abstract}

\maketitle

\section{Introduction}\label{s:intro}
An {\em abstract open book} is a pair $(\Si, \varphi)$, where $\Si$ is a connected, oriented surface with $\del\Si\neq\emptyset$ 
and the {\em monodromy} $\varphi$ is an element of the group $\Diff^+(\Si,\del\Si)$ of orientation--preserving diffeomorphisms of 
$\Si$ which restrict to the identity on a neighborhood of the boundary. We say that the monodromy $\varphi$ is 
{\em trivial} if it is isotopic to the identity of $\Si$ via diffeomorphisms which fix $\del\Si$ pointwise. 
Let $N_{\varphi}$ denote the mapping torus
\[
N_\varphi = \Si\x [0,1]/(p,1) \sim (\varphi(p),0).
\]
To the open book $(\Si,\varphi)$ one can associate a closed, oriented, connected 3--manifold $M_{(\Si,\varphi)}$ 
by using the natural identification of $\del N_\varphi = \del\Si\times S^1$ with the boundary of $\del\Si\times D^2$:  
\[
M_{(\Si,\varphi)} := N_\varphi\cup_{\del} \del\Si\x D^2.
\] 
The link $B:=\del \Si\x \{0\}\subset M_{(\Si,\varphi)}$ is fibered, with fibration $\pi\co M_{(\Si,\varphi)}\setminus B\to S^1$ 
given by the obvious extension of the natural projection 
\[
N_\varphi = \Si\x [0,1]/(p,1) \sim (\varphi(p),0)\to S^1=[0,1]/1\sim 0
\]
and monodromy equal to $\varphi$. In other words, the pair $(B,\pi)$ is an 
{\em open book decomposition} of $M=M_{(\Si,h)}$ with {\em binding} $B$, 
pages $\Si_\th:= \overline{\pi^{-1}(\th)}$, $\th\in S^1$ and monodromy $\varphi$. 
We will always identify $N_{\varphi}$ with the complement of a tubular neighborhood of $B$
in $M$. 

If $(B,\pi)$ is an open book decomposition of $M$ with page $\Si$, it is easy to see that 
$M$ has a Heegaard splitting of genus $b_1(\Si)$. Since $M$ is obtained from each  
handlebody of the splitting by attaching 2--disks and 3--balls, this immediately implies the inequality 
\begin{equation}\label{e:b1ineq}
b_1(M)\leq b_1(\Si).
\end{equation}
We will provide a refinement of Inequality~\eqref{e:b1ineq} with Proposition~\ref{p:intersection}.

The following theorem is our main result. Its proof is based on well--known 
results due to Reidemeister~\cite{Re33}, Singer~\cite{Si33} and Haken~\cite{Ha68} (see Section~\ref{s:proof}). 
Recall that each closed, oriented, connected 
3--manifold $M$ has a prime factorization, unique up to order of the factors, of the form 
\begin{equation}\label{e:factor}
M = M_1\#\cdots\# M_h\# S^2\times S^1\#\stackrel{(k)}{\cdots}\# S^2\times S^1,
\end{equation}
where each $M_i$ is irreducible (see~e.g.~\cite{He76}).

\begin{thm}\label{t:main}
Let $(B,\pi)$ be an open book decomposition 
of a closed, oriented, connected 3--manifold $M$ with page $\Si$ and monodromy $\varphi$. Then, $b_1(\Si)$ 
is equal to the number of $S^2\times S^1$--factors in the prime factorization of $M$ if and 
only if $\varphi$ is trivial and $M$ is a connected sum of $S^2\times S^1$'s.
\end{thm} 

Theorem~\ref{t:main} immediately implies the following corollary, which is also 
proved in~\cite[Proof of Theorem~1.3]{Ni10} and~\cite[Theorem~2]{GW13} using the fact 
that finitely generated free groups are not isomorphic to any of their nontrivial quotients.  

\begin{cor}\label{c:main} 
Any open book decomposition of $\#^k S^2\x S^1$ whose page $\Si$ satisfies $b_1(\Si)=k$ 
must have trivial monodromy.
\end{cor} 

Corollary~\ref{c:main} implies Corollary~\ref{c:bir-men}, 
which was obtained previously by Birman--Menasco as an application of their braid foliation 
techniques~\cite[Theorem~1]{BM92}. Grigsby and Wehrli gave two further proofs of Corollary~\ref{c:bir-men}, one using the fact 
that finitely generated free groups are not isomorphic to any of their nontrivial quotients, and the other 
using Khovanov homology~\cite{GW13}. 

\begin{cor}\label{c:bir-men}
Let $b \in B_n$ be a braid on $n$ strands such that its closure $\hat b$ is the trivial link $U_n$ 
with $n$ components. Then, $b$ is the identity. 
\end{cor}

\begin{proof}
Put $\hat b$ in braid form with respect to the binding of the trivial open book decomposition 
of $S^3$ and consider the two--fold branched cover $\Si(\hat b)$ along $\hat b$. Then, 
\[
\Sigma(\hat{b}) = \Sigma(U_n) = \#^{n-1} S^2 \times S^1.
\]
Pulling back the trivial open book of $S^3$ to $\Sigma(\hat{b})$ we obtain an open book decomposition of 
$ \#^{n-1} S^2 \times S^1$, whose page is a surface $\Si$ with $b_1(\Si)=n-1$, which we view as a 
2--fold branched cover of the disk with $n$ branch points. Under the identificaton of $B_n$ with the subgroup of the 
mapping class group of $\Si$ given by the elements commuting with the covering involution~\cite{BH73}, 
the monodromy of the open book is equal to $b$. By Corollary~\ref{c:main}, the braid $b$ must be 
the identity in $B_n$.
\end{proof}

Let $\Si$ and $\Si'$ be two orientable surfaces. By performing a boundary connected sum 
between them we obtain a surface $\Si\natural\Si'$.
If $\varphi$ is a diffeomorphism of $\Si$, $\psi$ is a diffeomorphism of 
$\Si'$ and both $\varphi$ and $\psi$ are the identity on a neighborhood of the  
boundary, we can form a diffeomorphism $\varphi \natural \psi$ of $\Sigma\natural\Si'$. 
This geometric operation yields a homomorphism 
\[ \Gamma_\Si \times \Gamma_{\Si'} \to \Gamma_{\Si\natural\Si'}, \]
which we will call {\em boundary connected sum homomorphism.}
A combination of Inequality~\eqref{e:b1ineq} with Corollary~\ref{c:main} yields the following 
Corollary~\ref{c:injectivity}, which can also be proved e.g.~applying~\cite[Corollary~4.2 (iii)]{PR00}.

\begin{cor}\label{c:injectivity}
Let $\Gamma_\Si$ be the mapping class group of the orientable surface $\Si$. 
Then, the boundary connected sum homomorphism 
\[
\Gamma_\Si \times \Gamma_{\Si'} \to \Gamma_{\Si\natural\Si'}
\] 
is injective.
\end{cor}

\begin{proof}
Under the map $(\Si,\varphi)\to M_{(\Si,\varphi)}$ described above, 
boundary connected sum of abstract open books corresponds to connected sum of 
$3$-manifolds:
\[
M_{(\Si\natural\Si', \varphi \natural \psi)} = M_{(\Si,\varphi)} \# M_{(\Si',\psi)}.
\]
Observe that $b_1(\Si\natural\Si') = b_1(\Si) + b_1(\Si')$. Therefore, 
if $\varphi \natural \psi$ is isotopic to the identity relative to the boundary then 
$M_{(\Si\natural\Si', \varphi \natural \psi)}$ is diffeomorphic to $\#^{b_1(\Si) + b_1(\Si')} S^2 \times S^1$. 
The uniqueness of the prime factorization for $3$--manifolds~\cite{He76} 
implies that $M_{(\Si,\varphi)} = \#^k S^2 \times S^1$ and $M_{(\Si',\psi)} = \#^l S^2 \times S^1$
for some non--negative integers $k, l$ such that $k+l = b_1(\Si)+b_1(\Si')$. By 
Inequality~\eqref{e:b1ineq} we have $k \le b_1(\Si)$ and $l \le b_1(\Si')$, which forces 
$k=b_1(\Si)$ and $l=b_1(\Si')$ as the only possibility. Corollary~\ref{c:main} implies 
that $\varphi$ and $\psi$ are isotopic to the identity. 
\end{proof}

The rest of the paper is organized as follows. In Section~\ref{s:further} we recall
two well known results independent of 
Theorem~\ref{t:main}, i.e.~Propositions~\ref{p:spheres-B} and~\ref{p:intersection}. Proposition~\ref{p:spheres-B} 
shows that any embedded 2--sphere disjoint from the binding of an open book decomposition 
is homologically trivial. Proposition~\ref{p:intersection} is a refinement of Inequality~\eqref{e:b1ineq} 
and can be viewed as saying that the homology of a closed, oriented, connected 3--manifold $M$ 
puts homological constrains on the monodromy of any open book decomposition of $M$. 
In Section~\ref{s:proof} we prove Theorem~\ref{t:main}. 

{\bf Acknowledgements:} the authors wish to thank the anonymous referees for valuable comments. 
The present work is part of the authors' activities within CAST, a Research Network
Program of the European Science Foundation. The first author was partially supported by the ERC grant ``Geodycon''. The second author was partially supported by 
the PRIN--MIUR research project 2010--2011 ``Variet\`a reali e complesse: geometria, 
topologia e analisi armonica''. 

\section{Non--separating 2--spheres and a refinement of Inequality~\eqref{e:b1ineq}}\label{s:further}

Given a closed, oriented, connected 3--manifold $M$ endowed with an 
open book decomposition $(B,\pi)$ and having a prime factorization as in~\eqref{e:factor}, 
one of the first questions one could ask is how a non--separating 2--sphere $S$ in $M$ can be 
positioned with respect to the binding $B$. Since $B$ is homologically trivial in $M$, 
the following proposition implies that, possibly after a small isotopy, each such $S$ 
must intersect $B$ transversally at least twice. 

\begin{prop} \label{p:spheres-B}
Let $(B,\pi)$ be an open book decomposition with page $\Si$ and monodromy $\varphi$ of a closed, 
oriented, connected 3--manifold $M$. Then, each embedded 2--sphere $S \subset M\setminus B$ bounds an embedded ball in $M \setminus B$ and, in particular, is homologically trivial in $M$.
\end{prop}

\begin{proof}
Recall that $M = N_{\varphi} \cup V$, where $V$ is a tubular neighborhood of the 
binding. Up to an isotopy of $S$, we can assume $S \subset N_\varphi$.  The universal cover of $N_\varphi$ is homeomorphic to $\R^3$ and from this the triviality of $[S]$ in
$H_2(M \setminus B)$, and therefore in $H_2(M)$, follows immediately. 

In order to prove that $S$ bounds a ball in $M \setminus B$ we need to use some
basic results in three--dimensional topology. In fact $\R^3$ is irreducible~\cite[Theorem~1.1]{Ha00} 
and this implies~\cite[Proposition~1.6]{Ha00} that $N_\varphi$ is also irreducible, 
therefore $S$ bounds an embedded ball in $N_\varphi$.
\end{proof}

We now establish a result which refines Inequality~\eqref{e:b1ineq}. Proposition~\ref{p:intersection} below can 
be viewed as saying that the homology of a closed, oriented, connected 3--manifold $M$ puts homological 
constraints on the monodromy of any open book decomposition of $M$.

For the rest of this section all homology groups will be taken with coefficients in the field $\Q$ of rational numbers
unless specified otherwise. Let $H_1(\Si, \partial \Sigma)^{\varphi}$ denote the subspace of $H_1(\Si,\del\Si)$ consisting 
of the elements fixed by the map 
\[
\varphi_* \colon H_1(\Sigma, \partial \Sigma) \to H_1(\Sigma, \partial \Sigma)
\] 
induced by the monodromy $\varphi\co\Si\to\Si$. 

\begin{prop}\label{p:intersection}
Let $(B,\pi)$ be an open book decomposition with page $\Si$ and monodromy $\varphi$ of a closed, oriented, connected 
3--manifold $M$. Then, 
\[
b_1(M) = \dim_\Q H_1(\Si,\del\Si)^\varphi.
\] 
More precisely, there is an isomorphism $H_2(M)\cong  H_1(\Sigma, \partial \Sigma)^\varphi$ induced by 
a well--defined map $H_2(M;\Z) \to H_1(\Sigma, \partial \Sigma;\Z)^\varphi$ given by $\al\mapsto [F\cap\Si]$, 
where $F\subset M$ is any closed, oriented and properly embedded surface which represents $\al$  
and intersects the page $\Sigma \times \{ 0 \}$ transversally.
\end{prop}

\begin{proof}
We can view $N_{\varphi}$ as the union of  $\Sigma \times [0, 1/2]$ and  $\Sigma \times 
[1/2, 1]$ with $(x, 1)$ identified to $(\varphi(x), 0)$. Using the fact that $\Sigma$ times an interval 
is homotopically equivalent to $\Si$, the (relative) Mayer--Vietoris sequence 
for this splitting gives the following exact sequence:
\[
H_2(\Sigma, \partial \Sigma)^2 \stackrel{f_1} \longrightarrow H_2(N_{\varphi}, \del N_{\varphi}) 
\stackrel{f_2} \longrightarrow H_1(\Sigma, \partial \Sigma)^2 
\stackrel{f_3} \longrightarrow H_1(\Sigma, \partial \Sigma)^2.
\]
The map $f_3$ is given by the matrix
\[
\left ( \begin{matrix} Id & Id \\ \varphi_* &  Id \end{matrix} \right )\in M_2(\End(H_1(\Si,\del\Si))).
\]
This immediately implies that the image of $f_2$ is isomorphic to $H_1(\Sigma, \partial \Sigma)^{\varphi}$.

Recall the decomposition $M = N_{\varphi} \cup V$, where $V$ is a tubular neighborhood of the 
binding. Since $H_2(V)=\{0\}$, the homology exact sequence for the pair $(M, V)$ implies that the map
$g\co H_2(M) \to H_2(M, V)$ induced by the inclusion map is injective. On the other hand, by excision the inclusion 
$N_\varphi \subset M$ induces an isomorphism $\psi\co H_2(N_{\varphi}, \partial N_{\varphi})\stackrel{\cong} \longrightarrow H_2(M, V)$. 
Moreover, it is easy to see that the image of the map $\psi\circ f_1$ maps injectively to $H_1(V)$ under 
the next map $\de\co H_2(M,V)\to H_1(V)$ in the exact sequence of the pair, while the image of $g$ maps trivially. 
This shows that the images of $f_1$ and of $\psi^{-1}\circ g$ have trivial intersection. 
Therefore the composition $f_2\circ\psi^{-1}\circ g$ sends $H_2(M)$ injectively into  
the image of the map $f_2$, which, as we have just shown, is isomorphic to $H_2(\Sigma, \partial \Sigma)^{\varphi}$.

We claim that $f_2\circ\psi^{-1}\circ g$ sends $H_2(M)$ also surjectively onto the image of $f_2$. 
In order to verify this, we argue by induction. Assume first that $\del\Si$ is connected. In this situation the map  
$\de\circ\psi\circ f_1$ is clearly surjective. Therefore, if $x\in H_2(N_\varphi, \del N_\varphi)$ with $f_2(x)\neq 0$, 
there exists $y\in H_2(\Sigma, \partial \Sigma)^2$ with $\de\circ\psi\circ f_1(y) = \de\circ\psi(x)$. It follows that    
setting $x' = x - f_1(y)$ we have $f_2(x') = f_2(x)$ and $\de\circ\psi(x') = 0$; therefore $x'$ is in 
the image of $\psi^{-1}\circ g$, and the claim is proved when $\del\Si$ is connected. 

Now assume $\del\Si$ is disconnected and denote by $|\del\Si|$  the number of its connected components. By the inductive 
hypothesis we assume that the claim holds for open books with $|\partial \Sigma|-1$ binding components. 
Let $(\widehat\Si, \widehat\varphi)$ be another abstract open book, constructed as follows. 
The connected, oriented surface $\widehat\Si$ is obtained 
by attaching a 2--dimensional 1--handle $h$ to $\del\Si$ so that $|\del\widehat\Si| = |\del\Si|-1$, while 
$\widehat\varphi$ is defined by first extending $\varphi$ as the identity over $h$, and then composing with a 
(positive or negative) Dehn twist along a simple closed curve in $\widehat\Si$ 
which intersects the cocore $c$ of $h$ transversely once.  It is a well--known fact that the open book 
decomposition $(\widehat B,\widehat\pi)$ associated to $(\widehat\Si, \widehat\varphi)$ is obtained 
from the open book decomposition $(B,\pi)$ associated to $(\Si, \varphi)$ by plumbing with a Hopf band, 
and that $M_{(\widehat\Si,\widehat\varphi)}$ is diffeomorphic to $M$ (see e.g.~\cite{GG06}). 
We can choose a basis $[c_1],\ldots, [c_{b_1(\Si)}]$ of $H_1(\Si,\del\Si)$ such that each $c_i\subset\Si$ is a properly 
embedded arc disjoint from $\ga\cap\Si$, and so that, viewing the classes $[c_i]$ in $H_1(\widehat\Si,\del\widehat\Si)$, 
when we add $[c]$ we obtain a basis of $H_1(\widehat\Si,\del\widehat\Si)$. Using this basis one can easily check that 
the natural inclusion map $H_1(\Si,\del\Si)\to H_1(\widehat\Si,\del\widehat\Si)$ restricts to an isomorphism
\[
H_1(\Si,\del\Si)^\varphi\cong H_1(\widehat\Si,\del\widehat\Si)^{\widehat\varphi}.
\]
Since $|\del\widehat\Si| = |\del\Si|-1$, by the inductive assumption we have 
$b_1(M) = \dim_{\Q} H_1(\widehat\Si,\del\widehat\Si)^{\widehat\varphi}$. This proves the claim in 
full generality. Finally, observe that the maps $f_2$ and $f_2\circ\psi^{-1}\circ g$ are well--defined over 
the integers. If we represent homology classes in $H_2(N_{\varphi}, \partial N_{\varphi};\Z)$ 
and $H_2(M;\Z)$ by oriented, properly embedded surfaces intersecting the page $\Sigma \times \{ 0 \}$ transversally 
and we follow the construction of the connecting homomorphism, we see that the maps 
$f_2$ and $f_2\circ\psi^{-1}\circ g$ are both realized geometrically by intersecting with $\Sigma \times \{ 0 \}$. 
This concludes the proof.
\end{proof}

\section{The proof of Theorem~\ref{t:main}}\label{s:proof} 

We start by recalling a basic result of Reidemeister and Singer about collections of compressing 
disks in a handlebody. We refer to~\cite{Jo06} for a modern presentation of this material. 
Let $H_g$ be a 3--dimensional handlebody of genus $g$. A properly embedded disk 
$D\subset H_g$ is {\em essential} if $\del D$ does not bound a disk in $\del H_g$. 

\begin{defn}\label{d:minsysdis}
A collection $\{D_1,\ldots, D_g\}\subset H_g$ of $g$ properly embedded, pairwise disjoint essential disks is a 
{\em minimal system of disks} for $H_g$ if the complement of a regular neighborhood of 
$\bigcup_i D_i$ in $H_g$ is homeomorphic to a 3--dimensional ball. 
\end{defn}

Let $D_1, D_2\subset H$ be properly embedded, essential disks in the handlebody $H_g$. 
Let $a\subset\del H$ be an embedded arc with one endpoint on $\del D_1$ and the other 
endpoint on $\del D_2$. Let $N$ be the closure of a regular neighborhood of $D_1\cup D_2\cup a$ 
in $H$. Then, $N$ is homeomorphic to a closed $3$--ball, and it intersects $\del H_g$ in a subset of 
$\del N$ homeomorphic to a three--punctured 2--sphere. The complement $\del N\setminus \del H_g$ 
of this subset consists of the disjoint union of three disks, two of which are isotopic to $D_1$ and $D_2$ respectively, 
and the third one is denoted by ${D_1}*_a D_2$. See Figure~\ref{f:disk-slide}.
\begin{figure}[ht]
\labellist
\hair 2pt
\pinlabel $D_1$ at 155 270
\pinlabel $D_2$ at 560 270
\pinlabel ${D_1}{*_a}D_2$ at 360 215
\pinlabel $a$ at 360 287
\endlabellist
\centering
\includegraphics[scale=0.4]{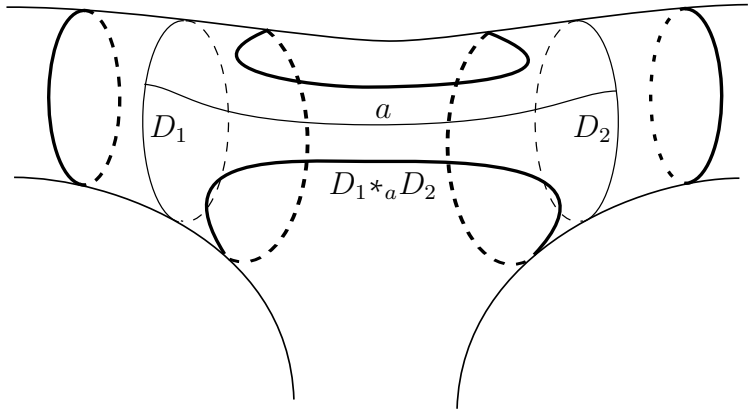}
\caption{A disk slide}
\label{f:disk-slide}
\end{figure}
Let ${\bf D}  = \{D_1,\ldots, D_g\}$ be a minimal system of disks for a handlebody 
$H_g$, $a\subset H_g$ an embedded arc with one endpoint on $\del D_i$, the other endpoint 
on $\del D_j$, with $i\neq j$, and the interior of $a$ disjoint from $\bigcup_i \del D_i$. Then, removing 
either $D_i$ or $D_j$ from ${\bf D}$ and adding $D_i *_a D_j$ yields a new minimal system 
of disks ${\bf D'}$ for $H_g$, well--defined up to isotopy~\cite[Corollary~2.11]{Jo06}. In this situation we say that ${\bf D'}$ is 
obtained from ${\bf D}$ by a {\em disk slide}. 

\begin{defn}\label{d:slideequiv}
Two minimal systems of disks for $H_g$ are {\em slide equivalent} if they 
are connected by a finite sequence ${\bf D}_1,\ldots,{\bf D}_m$ such that 
${\bf D}_{i+1}$ is obtained from ${\bf D}_i$ by a disk slide for each $i$. 
\end{defn}

To prove Theorem~\ref{t:main} we need the following result (see~\cite[Theorem~2.13]{Jo06} 
for a modern exposition). 

\begin{thm}[\cite{Re33, Si33}]\label{t:RS}
Any two minimal systems of disks for a handlebody are slide equivalent. 
\qed\end{thm} 

We can now start the formal proof of Theorem~\ref{t:main}. The first step is to normalize the 
position of certain non--separating $2$--spheres with respect to a Heegaard splitting. This will be done in the following lemma.

\begin{lemma}\label{standard spheres}
Let $M= H \cup H'$ be a Heegaard splitting of a $3$-manifold $M$ which admits a prime 
factorization 
\begin{equation}\label{e:factor2}
M = M_1\#\cdots\# M_h\# S^2\times S^1\#\stackrel{(k)}{\cdots}\# S^2\times S^1
\end{equation}
with $b_1(M)=k$. Then, there are pairwise disjoint, embedded 2--spheres $S_1, \ldots, S_k$ in $M$  such that each $S_i$ 
intersects the Heegaard surface $\partial H$ in a single circle $C_i$. Moreover, after choosing an orientation 
of each $S_i$, the corresponding 2--homology classes $[S_i]$ generate $H_2(M; \Q)$.
\end{lemma}

\begin{proof}
Suppose that $M'=M_1 \# \cdots \# M_h$ where each $M_i$ is irreducible. By definition  any 
embedded 2--sphere $S\subset M_i$ bounds a 3--ball. Therefore, if we denote by 
$S_1',\ldots, S_{h-1}'\subset M'$ the separating spheres along which the connected sums 
are performed and $S_h' \subset M'$ is any smoothly embedded 2--sphere disjoint from   
$S_1',\ldots, S_{h-1}'$, then the closure of some component of $M'\setminus\bigcup_{i=1}^h S_i'$ is a punctured 3--ball. 

In the terminology of Haken~\cite{Ha68}, a collection of pairwise disjoint, embedded 
2--spheres with such a property is called a {\em complete system of spheres}. Thus, 
the collection  $S_1',\ldots, S_{h-1}'$ is a complete system of spheres for $M'$.  
If we view each sphere $S_i'$ as contained in $M$ and 
denote by $S_{h-1+i}'\subset M$, for $i=1,\ldots, k$, the embedded 2--sphere
corresponding to $S^2\x \{ 1 \}$ in the $i$--th $S^2\x S^1$--factor of the factorization~\eqref{e:factor2}, 
the whole collection $S_1',\ldots, S_{h-1}', S'_h,\ldots, S_{h-1+k}'$ is a complete system of spheres for $M$. 

Observe that, since $b_1(M)=k$, $b_1(M')=0$. Then, after choosing orientations, the homology 
classes $[S_{h-1+i}']\in H_2(M;\Q)$ generate  $H_2(M;\Q)$ as a $\Q$--vector space, and {\it a fortiori} 
the same is true for the classes $[S_1'],\ldots, [S_{h-1+k}']$.

Now, according to the lemma on page 84 of~\cite{Ha68}, the system of spheres $S_1',\ldots, S_{h-1+k}'$ may be 
transformed by a finite sequence of isotopies and ``$\rho$--operations'' (see~\cite{Ha68} for the definition) 
into a collection of pairwise disjoint, incompressible 2--spheres $S_1,\ldots, S_t$, $t\geq h-1+k$, such 
that each $S_i$ intersects the Heegaard surface 
$\del H$ in a single circle $C_i = S_i\cap\del H$, and moreover  
the classes $[S_i]$ still generate $H_2(M;\Q)$. Since $\dim_\Q H_2(M;\Q)=k$, 
up to renaming the spheres we may assume that $[S_1],\ldots, [S_k]$ are generators 
of $H_2(M;\Q)$. This finishes the proof of the lemma.
\end{proof}


\begin{proof}[Proof of Theorem~\ref{t:main}]
Let $(B,\pi)$ be an open book decomposition 
of a closed, oriented, connected 3--manifold $M$ with page $\Si$ and monodromy $\varphi$. 
If $\varphi$ is trivial then it is easy to check that $M$ is homeomorphic to the connected sum
of $b_1(\Si)$ copies of $S^2\times S^1$. This proves one direction of the 
statement. For the other direction, suppose that $M$ factorizes as in~\eqref{e:factor}.  
In view of Proposition~\ref{p:intersection} or Inequality~\eqref{e:b1ineq} we have 
\[
b_1(\Si)\geq b_1(M) \geq k.
\]
If $b_1(\Si)=k$, the above inequality implies $b_1(M) = k$ and therefore if we set 
\[
M':= M_1\#\cdots\# M_h
\]
we have $b_1(M')=0$. 

Denote by $H_{b_1(\Si)}\subset M$ the handlebody of genus $b_1(\Si)$ consisting of a regular 
neighborhood of $\Si$ in $M$. Since $\Si$ is the fiber of a fibration, the closure of the complement 
$M\setminus H_{b_1(\Si)}$ is a handlebody as well, which we denote by $H'_{b_1(\Si)}$. It follows that 
$M$ admits the Heegaard splitting 
\begin{equation}\label{e:hdec}
M=H_{b_1(\Si)} \cup H'_{b_1(\Si)}.
\end{equation} 

By Lemma \ref{standard spheres} there are pairwise disjoint embedded spheres $S_1,
\ldots, S_k \subset M$ which generate $H_2(M;\Q)$ and such that each $S_i$ intersects the 
Heegaard surface $\partial  H_{b_1(\Si)}$ in a single circle $C_i$. 

Observe that each circle $C_i$ bounds the disk 
$D_i = S_i\cap H_{b_1(\Si)}$ inside $H_{b_1(\Si)}$ and the disk $S_i\cap H'_{b_1(\Si)}$ 
inside $H'_{b_1(\Si)}$. Since the map 
\[
H_2(M;\Q)\to H_1(\partial H_{b_1(\Si)};\Q)
\] 
appearing in the Mayer--Vietoris sequence associated with the decomposition~\eqref{e:hdec} 
is injective, after choosing orientations we see that the induced homology classes $[C_i]$ 
generate a half-dimensional subspace of $H_1(\partial H_{b_1(\Si)};\Q)$ which is Lagrangian
for the intersection form on $H_1(\partial H_{b_1(\Si)};\Q)$ because the $C_i$'s are pairwise disjoint. 

We now claim that the $D_i$'s are a minimal system of compressing disks for $H_{b_1(\Si)}$.
To see this we can argue by induction on $b_1(\Si)$. If $b_1(\Si)=0$ there is nothing to prove, so 
we may assume $b_1(\Si)>0$.  Let $N$ be an open regular neighborhood of $D_1$.
Since $[C_1]\neq 0$, $H_{b_1(\Si)}\setminus N$ is connected and therefore by e.g.~\cite[Proposition~5.18]{Jo06} 
it is a handlebody. Moreover, the remaining homology classes $[C_i]$, $i\geq 2$, 
generate a Lagrangian subspace in the first homology group of the boundary of $H_{b_1(\Si)}\setminus N$. 
By the inductive assumption the disks $D_i$, for $i\geq 2$, are a minimal system of compressing 
disks for $H_{b_1(\Si)}\setminus N$, which proves the claim.  


Recall that, by construction, the curves $C_i = \del D_i$ bound compressing 
disks in $H'_{b_1(\Si)}$. Arguing as for $H_{b_1(\Si)}$ shows that such disks constitute a minimal system
for $H'_{b_1(\Si)}$. Thus, surgering $M$ along the spheres $S_1,\ldots, S_k$  yields a 3--manifold having 
a genus--0 Heegaard splitting, i.e.~$S^3$. This implies that $M$ 
is a connected sum of $k$ copies of $S^2\times S^1$, and we are left to show that 
the monodromy $\varphi$ is trivial. 

Now we choose a system of arcs for $\Si$, i.e.~a collection of properly embedded, pairwise disjoint 
oriented arcs $a_1,\ldots, a_{b_1(\Si)}\subset\Si$ 
whose associated homology classes $[a_i]\in H_1(\Si,\del\Si;\Q)$ generate the 
$\Q$--vector space $H_1(\Si,\del\Si;\Q)$. Then, after fixing an identification  
$H_{b_1(\Si)} = \Si\times I$, the disks $a_i\times I\subset \Si\times I$ yield 
another minimal system of disks $\{D'_i\}_{i=1}^g$ for $H_{b_1(\Si)}$. Thus,   
according to Theorem~\ref{t:RS}, the system $\{D_i\}_{i=1}^g$ is slide equivalent to 
the system $\{D'_i\}_{i=1}^g$. But recall that, by construction, each curve $C_i = \del D_i$ 
bounds a compressing disk in $H'_{b_1(\Si)}$, and a moment's reflection shows that any disk slide among the $D_i$'s 
gives rise to a disk $D_i *_a D_j$ whose boundary also bounds a compressing disk in $H'_{b_1(\Si)}$. 
By induction we conclude that any minimal system of disks $\{\tilde D_i\}_{i=1}^g$ 
obtained from $\{D_i\}_{i=1}^g$ by a finite sequence of isotopies and disk slides still has the property that 
each curve $\del\tilde D_i$ bounds a compressing disk in $H'_{b_1(\Si)}$. 

In particular, this conclusion 
applies to the system $\{D'_i\}_{i=1}^g$, showing that each of the circles 
$\del D'_i$ bounds a compressing disk in $H'_{b_1(\Si)}$.  Since the splitting~\eqref{e:hdec} 
is induced by the open book decomposition $(B,\pi)$, 
we can choose an identification $H'_{b_1(\Si)} = \Si\times [0,1]$ such that each $\del D'_i$ is 
of the form 
\[
a_i\times\{0\}\bigcup \varphi(a_i)\times\{1\}, 
\]
where $\varphi$ is the monodromy 
of $(B,\pi)$. The fact that $\del D'_i$ bounds a disk in $H'_{b_1(\Si)}$ says that there is 
a family of arcs in $\Si\times I$ interpolating between $a_i\times\{0\}$ and $\varphi(a_i)\times\{1\}$. 
Mapping such family to $\Si$ via the projection $\Si\times I\to\Si$ shows that each $a_i$ is homotopic 
to $\varphi(a_i)$ (with fixed endpoints), and therefore by~\cite{Ep66} each $a_i$ is isotopic to $\varphi(a_i)$ 
via an isotopy which keeps the endpoints fixed. Since $\{a_i\}$ is a system of arcs for $\Si$, 
a standard argument based on the Alexander lemma~\cite[Lemma~2.1]{FM11} implies that 
$\varphi$ is isotopic to the identity of $\Si$ via diffeomorphisms which fix $\del\Si$ pointwise. 
This concludes the proof of Theorem~\ref{t:main}.
\end{proof}

\bibliographystyle{amsplain}
\bibliography{biblio}
\end{document}